\documentclass[12pt,oneside,reqno,twoside]{article}

\usepackage{geometry,verbatim}
\usepackage{anyfontsize}

\usepackage{enumerate}
\usepackage[shortlabels]{enumitem}

\usepackage{
mathtools,amsmath,amssymb,enumitem,amsthm,nccmath
}
\usepackage{graphicx,url}

\usepackage{xfrac} % sideway fractions
\usepackage{nicefrac}

\usepackage{color,cite,verbatim}

\usepackage[svgnames]{xcolor}

\usepackage[colorlinks=true, linkcolor=Maroon, urlcolor=Maroon]{hyperref}

\hypersetup{citecolor=blue}

\mathtoolsset{showonlyrefs} %%% Only number referenced equations

\usepackage{dsfont}
\usepackage{bm}
\usepackage{mathrsfs} %For \mathscr{S} -> Schwartz
\usepackage{suetterl} %%frak letters%%
\usepackage{caption}
\usepackage{subcaption}		%%subfigures, etc.%%
\usepackage{braket} %Bra-Ket notation
\usepackage{stmaryrd} %%double brackets

\usepackage[T1]{fontenc}
\usepackage{currvita}

\usepackage{csquotes} %% quotes

\usepackage{empheq}
\usepackage[most]{tcolorbox}

\usepackage{orcidlink}
\usepackage[affil-it]{authblk}

\usepackage{float}
\usepackage[all]{xy}
\usepackage{tikz}
\usepackage{pgfplots}
\usetikzlibrary{arrows,arrows.meta,backgrounds,positioning,fit}
\usepgfplotslibrary{fillbetween}
\usetikzlibrary{patterns}
\pgfplotsset{width=7cm,compat=newest}
\tikzset{%
  dots/.style args={#1per #2}{%
    line cap=round,
    dash pattern=on 0 off #2/#1
  }
}

\allowdisplaybreaks[1]		%% page break in equation arrays %%

\def\N{{\mathbb N}}	
\def\Q{{\mathbb Q}}				%%%%%%%%%
\def\R{{\mathbb R}}				 %% sets %%%%
\def\Z{{\mathbb Z}}				%%%%%%%%%

\def\cF{{\mathcal F}}

\renewcommand{\Re}{\mathrm{Re}}
\def\1{{\bf 1}}

\newcommand{\floor}[1]{\left\lfloor #1 \right\rfloor}

\def\dist{\operatorname{dist}}

\def\dx#1{\mathrm{d}#1\ }

\renewcommand{\Re}{\mathrm{Re}}

\makeatletter
\def\namedlabel#1#2{\begingroup					
	#2%
	\def\@currentlabel{#2}%
	\phantomsection\label{#1}\endgroup
}
\makeatother

\theoremstyle{plain}
\newtheorem{theorem}{Theorem}[section]		
	
\newtheorem{proposition}[theorem]{Proposition}		   %% Numbering%

\newtheorem{lemma}[theorem]{Lemma}

\newtheorem{remark}[theorem]{Remark}
\newtheorem*{remark*}{Remark}

\numberwithin{equation}{section} 

\newtheoremstyle{case}{}{}{}{}{}{:}{ }{}
\theoremstyle{case}

\usepackage{scalerel}[2014/03/10]
\usepackage[usestackEOL]{stackengine}
\def\avint{\,\ThisStyle{\ensurestackMath{%
  \stackinset{c}{.2\LMpt}{c}{.5\LMpt}{\SavedStyle-}{\SavedStyle\phantom{\int}}}%
  \setbox0=\hbox{$\SavedStyle\int\,$}\kern-\wd0}\int}

\def\vep{{\varepsilon}}
\def\young{A_{\text{Y}}}

\def\heat{A_{\text{par}}}
\def\GNS{A_{\text{GNS}}}

\makeatother

\begin{document}

	\title{Upper bounds on homogeneous fractional Gagliardo-Nirenberg-Sobolev constants via parabolic estimates}
 
  \author{
Michael Hott\orcidlink{0000-0003-4243-6585}%
	\thanks{E-mail: \texttt{mhott@umn.edu}}}

    % \subjclass[2020]{46B70, 46E35}
    % \keywords{Functional inequalities, Sobolev spaces, interpolation}
    
    \maketitle
	
	\begin{abstract}
        Common proofs of the Gagliardo-Nirenberg-Sobolev (GNS) do not provide explicit bounds on the involved constants, unless a sharp constant is being determined. GNS inequalities naturally occur in error estimates for numerical approximations. In particular, bounds on GNS constants allow us to provide explicit a priori estimates. We provide an algorithm that determines upper bounds on the non-endpoint homogeneous GNS inequalities in terms of explicit upper bounds for Young's convolution inequality and parabolic estimates. Our method is based on the heat-kernel representation of the inverse Laplacian, from which we deduce interpolation estimates.
	\end{abstract}
	\paragraph{Statements and Declarations} The authors do not declare financial or non-financial interests that are directly or indirectly related to the work submitted for publication.
	
	\paragraph{Data availablity} The manuscript has no associated data.
	 \paragraph{Acknowledgments} My research was supported by 'University of Texas at Austin Provost Graduate Excellence Fellowship' and NSF grants DMS-1716198 and DMS-2009800 through Prof. Thomas Chen (UT Austin). I want to thank Bill Beckner, Haim Brezis, Dmitriy Bilyk, Svitlana Mayboroda, Nata\v{s}a Pavlovi\'c, Petru Mironescu, Petronela Radu, Pablo Stinga for pointing out helpful references. I want to thank Hao Jia and Lutz Weis for stimulating discussions.

	\tableofcontents

	\section{Introduction}
    
    Let $d\in\N$ denote the space dimension. Let $s,s_1,s_2\in\R$, $p,p_1,p_2\in[1,\infty]$ be such that
    \begin{equation}\label{ass-para}
        \frac1{p_2}-\frac{s_2}d \, < \, \frac1p -\frac{s}d \, < \, \frac1{p_1}-\frac{s_1}d \, ,
    \end{equation}
    and $\theta\in(0,1)$ be given by
    \begin{equation} \label{def-theta}
        \frac1p -\frac{s}d \, = \, \theta \Big(\frac1{p_1}-\frac{s_1}d\Big) +(1-\theta) \Big(\frac1{p_2}-\frac{s_2}d\Big) \, .
    \end{equation}
    For this range of parameters, we consider the homogeneous \emph{Gagliardo-Nirenberg-Sobolev} (GNS) inequality
    \begin{equation}\label{GNS}
        \||\nabla|^s f\|_{L^p(\R^d)} \, \lesssim \, \||\nabla|^{s_1}f\|_{L^{p_1}(\R^d)}^{\theta}\||\nabla|^{s_2}f\|_{L^{p_2}(\R^d)}^{1-\theta} \, .
    \end{equation}
    % From now on, we will abbreviate Lebesgue norms by $\|\cdot\|_p$ and fix the space to be $\R^d$.
    This interpolation inequality goes back to Gagliardo \cite{gagliardo59} and Nirenberg \cite{nirenberg59}, where they proved the case $s,s_1,s_2\in\N_0$, and where they assume
    \begin{align} \label{ass-GN}
        s \, = \, \theta s_1 + (1-\theta) s_2 \, , \quad \frac1p \, = \, \frac{\theta}{p_1} + \frac{1-\theta}{p_2} \, ,
    \end{align}
    where $\theta\in(0,1)$, without the restriction \eqref{ass-para}. The original proofs by Nirenberg and Gagliardo are based on establishing an interplation inequality on intervals in dimension 1, for $s=1$, $s_1=0$ and $s_2=2$, and then generalized to arbitrary $s\in\N$ via induction, see, e.g., \cite{fiorenza2021detailed} for a review. One can also prove \eqref{GNS} via real or complex interpolation, covering, e.g., the cases $0\leq s_1<s_2<\infty$, $p_1=p_2=p=1$, see \cite[Section 7.32]{adams-fournier-Sobolev}. Other proofs employ wavelets, see, e.g., \cite{cohen-wavelet-00,cohen-dahmen-daubechies03}. \eqref{GNS} can also be formulated in Besov spaces \cite{hajaiej11}.
    \par Brezis and Mironescu have studied sufficient and necessary condition for the inhomogeneous equation associated with \eqref{GNS} to hold, see \cite{brezis-miro-GN-2018}. They assume $s,s_1,s_2\geq0$, $p,p_1,p_2\in[1,\infty]$ with the interpolation assumption \eqref{ass-GN}. Then they establish the necessary and sufficient condition that 
    \begin{equation}\label{eq-brezis-nec-suff}
        s_2 \in \N \, , \, p_2=1 \, , \, 0 \, < \, s_2-s_1 \, \leq \, 1-\frac1{p_1} 
    \end{equation}
    does not hold in order for
    \begin{equation} \label{GNS-inh}
        \|f\|_{W^{s,p}} \, \lesssim \, \|f\|_{W^{s_1,p_1}}^\theta \|f\|_{W^{s_2,p_2}}^{1-\theta}
    \end{equation}
    to be true. In a subsequent work, they studied the case $s,s_1,s_2\geq 0$, $p,p_1,p_2\in[1,\infty]$ satisfying \eqref{def-theta}, with the additional restriction
    \begin{equation}\label{ass-BM}
        s \, < \, \theta s_1 + (1-\theta)s_2 \, , \, \quad (s_1,p_1)\neq(s_2,p_2) \, .
    \end{equation}
    Then they proved that \eqref{GNS} is valid unless either of the following conditions is satisfied:
    \begin{enumerate}
        \item $d=1$, $s_2\in\N_0$, $1<p_1\leq\infty$, $p_2=1$, $s_1=s_2-1+\frac1{p_1}$, and either $[1<p_1<\infty,s=s_2-1]$ or $[s_2+\frac\theta{p_1}-1 < s < s_2 +\frac{\theta}{p_1}-\theta]$ \, .
        \item $d\geq 1$, $s_1<s_2$, $s_1-\frac{d}{p_1}=s_2-\frac{d}{p_2}=s \in\N_0$, $p=\infty$, $(p_1,p_2)\neq(\infty,1)$ (for every $\theta\in(0,1)$).
        \item $s_1\leq s\leq s_2$
        \begin{enumerate}
            \item $d=1$, $s_2\in\N$, $1<p_1<\infty$, $p_2=1$, $s_1=s_2-1+\frac1{p_1}$, $s_2+\frac\theta{p_1}-1<s<s_2+\frac\theta{p_1}-\theta$, and $s\geq s_1$.
            \item $d\geq 1$, $p_1=\infty$, $1<p_2<\infty$, $p=\infty$, $s_1=s\in\N_0$, $s_2=s+\frac{d}{p_2}$ (for every $\theta\in(0,1)$).
        \end{enumerate}
    \end{enumerate}
    Observe that \eqref{ass-para} excludes all of these cases. 
    % \par None of the previously mentioned approaches establish explicit bounds on the GNS constants. 
    \par For some special choices of parameters of $s,s_1,s_2$, $p,p_1,p_2$, there are many works establish \emph{sharp} constants. In the case $p_1=p_2=2$, $s_1=1$, with $|\nabla|$ replaced by $\nabla$, $s=s_2=0$, $0<p<\frac{2d}{d-2}$, Weinstein \cite{weinstein82} recognized the relationship between the GNS inequality and a sharp criterion for the global wellposedness of the nonlinear Schr\"odinger equation 
    \begin{equation}
        i\partial_t \phi \, = \, -\frac12\Delta \phi  - |\phi|^{2\sigma}\phi \, ,
    \end{equation}
    where $\sigma=\frac{p-2}2$. Using symmetric rearrangements, Talenti \cite{talenti76} proved the sharp Sobolev inequality in case $p\in(1,\infty)$, which corresponds to the endpoint case $\theta=1$, $s=s_2=0$, $s_1=1$. Another approach relates the optimal GNS inequality with optimal transport, see \cite{agueh2006sharp,cordero-villani-02}, in the range $p_1,p_2,p\in(1,\infty)$, $s=s_2=0$, $s_1=1$. Other approaches for this range of parameters are based on PDEs \cite{agueh2008sharp}, and other methods \cite{beckner02,bolley2020new,delpinodolbeatGN02,delpino-dolbeaut-logsob-03,nguyen2021}. The optimal constant for integer $s,s_1\in\N_0$, $s_2=0$ with $p_1=p_2=2$ has been determined in \cite{morosi2018constants}; they also include the endpoint case $\theta=1$. For works studying the sharp GN inequality involving fractional derivatives, we refer to \cite{bellazzinifrank14,franklieb2012,franksilvestre,fernandez22,Zhang21}. For related works, we refer to \cite{bellazzini2018sharp,chen2011sharp,dai23,dolbeault2014one,dolbeaut2018,lundholm2018methods,martin2007sharp,morosi2018constants,nguyen2015sharp,riviereStrzelecki05,van2022fractional}. 
    \par Apart from the cases for which optimal constants are known, upper constants are usually not explicitly provided in the known proofs. The goal of this work is to provide a unified framework with an algorithm to determine upper bounds on the non-endpoint GNS constants.

    \subsection{Application to approximation error estimates}

    Before turning to our results, we shall provide two instances for which knowledge of bounds on GNS constants provides an order of magnitude for approximation errors.
    \paragraph{Discretization error in Riemann sums}

    In the following, denote
    \begin{equation}\label{def-lattice-const}
        \Xi(d,\delta) \, := \, \sum_{n\in\Z^d\setminus\{0\}}\frac1{|n|^{d+\delta}} \quad \forall d\in\N, \, \delta>0 \, .
    \end{equation}

    In \cite[Lemma 5.5]{chenhott2023}, the following result was proved for $d=3$, see also \cite{Trefethen2014}.
    
    \begin{proposition}[Riemann sum rule]\label{prop-riem-discr.}
        Let $\vep,\delta>0$, $d\in\N$. Then we have for all $f\in L^1(\R^d)\cap W^{1,d+\delta}(\R^d)$ that
        \begin{align}
            \Big|\vep^d\sum_{x\in \Z^d}f(\vep x) \, - \, \int_{\R^d}\dx{x}f(x)\Big| \, \leq \, \frac{\Xi(d,\delta)}{(2\pi)^{d+\delta}}\||\nabla|^{d+\delta}f\|_1 \vep^{d+\delta} \, .
        \end{align}
    \end{proposition}

    We have the following bound on $\Xi(d,\delta)$.

    \begin{lemma}\label{lem-lattice-const-bd}
        Let $\zeta$ denote the Riemann $\zeta$-function. For any $d\in\N$, $\delta>0$, we have that
        \begin{equation}
            \Xi(d,\delta) \, \leq \, \sum_{k=1}^d\binom{d}{k}2^k \frac1{k^{\frac{d+\delta}2}}\zeta\Big(\frac{d+\delta}k\Big)^k \, .
        \end{equation}
    \end{lemma}
    For convenience, we present proofs of both, Prop. \ref{prop-riem-discr.} and Lemma \ref{lem-lattice-const-bd} in Appendix \ref{sec-error-est}.
    \par In applications, $f$ often is a concatenation of functions and we would like to express a bound terms of norms of the concatenated functions instead. For integer derivatives, one can use the Fa\`{a} di Bruno formula, see, e.g., \cite{turcu2020vectorFdB}. In the case of fractional derivatives, there is currently no generalization, see, e.g., \cite{Podlubny1999fractional,Tarasov2016Fractional}. Using the GNS inequality instead, one can interpolate between integer derivatives, to which one can apply the regular Fa\`{a} di Bruno formula. 
    \par One may wonder why in Proposition one does not take $d+\delta$ to be integer, in order to apply the regular chain rule. One reason for which it can be desirable to start with a fractional derivative instead, is that for rough choices of $f$ the norms can deteriorate badly. After interpolating with the GNS inequality, the worse integer derivative comes with an additional power which is less than 1. More precisely, let $r\in(0,\infty)$, and let $\floor{x}$ denote the smallest integer larger than $x$, and $\{x\}=x-\floor{x}$ denote the fractional part. By the GNS inequality \eqref{GNS}, we have that
    \begin{equation}
        \||\nabla|^rf\|_1 \, \lesssim \, \|(-\Delta)^{\floor{\frac{r}2}}f\|_1^{1-\{\frac{r}2\}}\|(-\Delta)^{\floor{\frac{r}2}+1}f\|_1^{\{\frac{r}2\}} \, .
    \end{equation}

    \paragraph{Error in ergodic approximation}
    In ergodic theory, the \emph{Birkhoff ergodic theorem} plays a central role, see, e.g., \cite{KerrLi-erg-th}. Let $\alpha\in[0,1)\setminus \Q$ and $f\in L^1_{\mathrm{loc}}(\R)$ be $\Z$-periodic. Then the Birkhoff ergodic theorem states that for almost every $x_0\in[0,1]$ 
    \begin{equation}
        \lim_{n\to\infty}\frac1n\sum_{k=1}^nf(x_0+k\alpha) \, = \, \int_0^1\dx{x}f(x) \, .
    \end{equation}
    In order to obtain a rate of convergence, one needs to impose a \emph{Diophantine condition} on $\alpha$. We say that $\alpha$ is $(\sigma,K)$-\emph{Diophantine} with $K>0$, $\sigma\geq0$ iff for any $n\in\Z$
    \begin{equation}\label{def-diophantine}
        \dist(\alpha n,\Z) \, \geq \, \frac{K}{|n|^{1+\sigma}} \, .
    \end{equation}
    Lebesgue-almost every $\alpha\in[0,1]$ is Diophantine. Then we have the following convergence rate result, see also \cite{cazeaux-luskin-CB-2017}.
    
    \begin{proposition}[Quantitative ergodic theorem]\label{prop-erg-thm}
        Let $\alpha\in[0,1)$ be $(K,\sigma)$-Diophantine and $f\in L^1_{\mathrm{loc}}(\R)$ be $\Z$-periodic. Denote $\hat{f}(p):=\int_0^1\dx{x}e^{-2\pi i xp}f(x)$ and assume that $|\cdot|^{1+\sigma}\hat{f}\in \ell^1(\Z\setminus\{0\})$. Then we have that
        \begin{equation}
            \Big|\frac1n\sum_{k=1}^nf(x_0+k\alpha) \, - \, \int_0^1\dx{x}f(x)\Big| \, \leq \, \frac{\||\cdot|^{2+\sigma+\delta}\hat{f}\|_{\ell^1(\Z\setminus\{0\})}}{2Kn} \, .
        \end{equation}
    \end{proposition}

    \begin{remark}
        Following the ideas of the proof of Prop. \ref{prop-riem-discr.}, one can show that for any $\delta>0$
        \begin{equation}
            \||\cdot|^{1+\sigma}\hat{f}\|_{\ell^1(\Z\setminus\{0\})} \, \leq \, \Xi(1,1+\delta)\||\nabla|^{2+\sigma+\delta}f\|_{L^1([0,1])} \, .
        \end{equation}
        In particular, it suffices in Proposition \ref{prop-erg-thm} to assume that $f\in W^{2+\sigma+\delta}([0,1])$. We then proceed as above to interpolate the fractional Sobolev norms by integer Sobolev norms.
    \end{remark}

    For convenience, we also present a proof of Prop. \ref{prop-erg-thm} in Appendix \ref{sec-error-est}.

    \section{Preliminaries and results}
    
    In our estimates, Young's inequality in its quantitative version will play a crucial role. Let $1\leq p,q,r\leq \infty$ be such that 
    \begin{equation}\label{eq-young-cond}
            \frac1q+\frac1r=1+\frac1p \, .
    \end{equation}
    Throughout this work let $\infty^0:=1$. Define 
    \begin{equation}\label{def-young-const}
            \young(p,q,r,d) \, = \, \begin{cases}\Big(\frac{(p')^{1/p'}q^{1/q}r^{1/r}}{p^{1/p}(q')^{1/q'}(r')^{1/r'}}\Big)^{d/2} \, , & \mbox{\ if\ } 1<p,q,r<\infty \, ,\\
            1  & \, \mbox{\ else}  \, .
            \end{cases}
    \end{equation}
    Then Beckner \cite{beckner1975} and Brascamp-Lieb \cite{brascamplieb1976} have proved that, when $1<p,q,r<\infty$
    \begin{equation}\label{eq-young}
        \|f*g\|_p \, \leq \, \young(p,q,r,d)\|f\|_q\|g\|_r 
    \end{equation}
    and that $\young(p,q,r,d)$ is optimal. The inequality also holds for the end-points $1\leq p,q,r\leq \infty$ satisfying \eqref{eq-young-cond}, see, e.g., \cite{liebloss}. 
    
    \par Let $\Gamma$ denote the Gamma function. Let $1\leq r\leq p \leq \infty$, and let $q\in[1,\infty]$ be given by
    \begin{equation}
        1+\frac1p \, = \, \frac1q+\frac1r \, .
    \end{equation}
    Let $\{x\}$ and $\floor{x}$, respectively, denote the fractional and integer part of $x\in\R$, respectively. First let $s\geq0$.
    \par If $\{s/2\}=0$, we define, in case $r<p$,
    \begin{align} \label{def-heat-int-r<p}
        \begin{split}
        \MoveEqLeft \heat(p,r,s,d) \\
        & := \, \frac{\young(p,q,r,d)2^s}{(4\pi)^{\frac{d}2(\frac1r-\frac1p)}q^{\frac{d}{2q}}(q')^{\frac{d}{2q'}}}\frac{\Gamma(\frac{d+s}2)}{\Gamma(\frac{d}2)s^{\frac{s}2}} \Big[\frac{s}2+\frac{d}2\Big(\frac1r-\frac1p\Big)\Big]^{\frac{s}2+\frac{d}2(\frac1r-\frac1p)} \, ,
        \end{split}
    \end{align}
    and, if $r=p$,
    \begin{equation}\label{def-heat-int-r=p}
        \heat(p,p,s,d) \, := \, \young(p,1,p,d)\frac{\Gamma(\frac{d+s}2)2^{\frac{s}2}}{\Gamma(\frac{d}2)} \, .
    \end{equation}
    If $\{s/2\}>0$, define
    \begin{align}\label{def-heat-frac}
        \begin{split}
            \MoveEqLeft\heat(p,r,s,d)\\
            :=& \, \heat\Big(p,r,2\floor{\frac{s}2}+2,d\Big) \, \frac{\Gamma\big(\frac{s}2+1+\frac{d}2(\frac1r-\frac1p)\big)}{\Gamma\big(\floor{\frac{s}2}+1+\frac{d}2(\frac1r-\frac1p)\big)} \, .
        \end{split}
    \end{align}
    \par Now let $s<0$ and assume
    \begin{equation}
        \frac1r \, > \, \frac1p -\frac{s}d \, .
    \end{equation}
    Define
    \begin{align} \label{def-heat-inv<inf}                  \heat(p,r,s,d)
            := \, \frac{\young(p,q,r,d)}{(4\pi)^{\frac{d}2(\frac1r-\frac1p)}q^{\frac{d}{2q}}}\frac{\Gamma(\frac{d}2(\frac1r-\frac1p)+\frac{s}2)}{\Gamma(\frac{d}2(\frac1r-\frac1p))} \, .
    \end{align}
    
    \par With this notation, we prove the following result.

        \begin{theorem}[Parabolic estimates]\label{thm-heat-ST}
            Let $f\in C^\infty_c(\R^d)$, and $t>0$. Let $s\in\R$ and $1\leq r\leq p \leq\infty$ be such that either $s\geq0$ or
            \begin{equation}        0 \, < \, -s \, < \, d\Big(\frac1r \, - \, \frac1p\Big) \, .
            \end{equation}
            Then we have that
            \begin{equation}
            \||\nabla|^s e^{t\Delta}f\|_p \,
            \leq \, \heat(p,r,s,d) t^{-\frac{s}2-\frac{d}2(\frac1r-\frac1p)}\|f\|_r \, .
            \end{equation}
        \end{theorem}

        \begin{remark}
            Theorem \ref{thm-heat-ST} is a standard result known to experts working with parabolic equations. We list the result here as we were not able to find a reference that provides explicit constants.
        \end{remark}

        \begin{remark}\label{rem-sob}
            The case $-s= d(\frac1r-\frac1p)$ would correspond to a Sobolev inequality. The methods applied in this work did not allow us to treat this case. In fact, observe that the involved constant diverges as
            \begin{equation}
                \Gamma
                \Big(\frac{d}2\Big(\frac1r-\frac1p\Big)+\frac{s}2\Big) \sim \Big(\frac{d}2\Big(\frac1r-\frac1p\Big)+\frac{s}2\Big)^{-1}
            \end{equation}
            as $-s\nearrow\frac{d}2(\frac1r-\frac1p)$, and that we are interested in the limit $(-s,t)\to (\frac{d}2(\frac1r-\frac1p)^-,0^+)$ along a suitable subsequence $((-s_n,t(s_n)))_{n\in\N}$.
            \par Instead, we refer to, e.g., \cite{franklieb2012} for a class of sharp fractional Sobolev inequalities.
        \end{remark}

        For the second statement, let $s,s_1,s_2\in\R$, $p,p_1,p_2\in[1,\infty]$ be such that
        \begin{equation}
            \frac1{p_2}-\frac{s_2}d \, < \, \frac1p -\frac{s}d \, < \, \frac1{p_1}-\frac{s_1}d \, ,
        \end{equation}
        see \eqref{ass-para}, and $\theta\in(0,1)$ be given by
        \begin{equation}
            \frac1p -\frac{s}d \, = \, \theta \Big(\frac1{p_1}-\frac{s_1}d\Big) +(1-\theta) \Big(\frac1{p_2}-\frac{s_2}d\Big) \, ,
        \end{equation}
        see \eqref{def-theta}.
        Define
        \begin{align}\label{def-Sigma}
            \begin{split}
                \MoveEqLeft \Sigma \, := \, \Sigma(p_1,p_2,p,s_1,s_2,s,d) \, := \\
                & \, \Bigg\{ (\beta_1,\beta_2,r_1,r_2,q_1,q_2,\sigma) \in (\theta,1)\times(0,\theta)\times[1,\infty]^4\times(0,\infty)\Big) \mid \\
                & \, \sigma \, > \, \frac{s_2-s}2 \, - \, \frac{d}{2p_2} \, ,\\
                & \, \frac1p \, - \, (1-\beta_1)\Big(\frac1{p_{2}}-\frac{s_{2}-s-2\sigma}d\Big) \, < \, \frac{\beta_1}{r_1} \, < \, \beta_1\Big(\frac1{p_1}-\frac{s_1-s-2\sigma}d\Big) \, , \\
                & \, \frac1p \, - \, \beta_2\Big(\frac1{p_1}-\frac{s_1-s-2\sigma}d\Big) \, < \, \frac{1-\beta_2}{r_2} \, < \, (1-\beta_2)\Big(\frac1{p_2}-\frac{s_2-s-2\sigma}d\Big) \, , \\
                & \, \frac1p \, = \, \frac{\beta_1}{r_1} \, + \, \frac{1-\beta_1}{q_1} \, = \, \frac{1-\beta_2}{r_2} \, + \, \frac{\beta_2}{q_2} \Bigg\} \\
                & \, \cup \{1\}\times\{0\}\times\{p\}^2\times [1,\infty]^2\times \Big(\frac{s_2-s}2 \, - \, \frac{d}{2p_2},\infty\Big) \, .
            \end{split}
        \end{align}
        Assumption \eqref{ass-para} ensures that $\Sigma\neq\emptyset$.
        Define
        \begin{align}
            \begin{split}
                \MoveEqLeft\GNS(p,p_1,p_2,s,s_1,s_2,d)\\
                :=& \, \inf_{(\beta_1,\beta_2,r_1,r_2,q_1,q_2,\sigma)\in\Sigma}\frac4{d \, \Gamma(\sigma)\big(\frac1{p_1}-\frac1{p_2}-\frac{s_1-s_2}d\big)}\\
                & \, \Bigg(\frac{\heat(q_2,p_1,s+2\sigma-s_1,d)^{\beta_2}\heat(r_2,p_2,s+2\sigma-s_2,d)^{1-\beta_2}}{\theta-\beta_2}\Bigg)^{\frac{\theta-\beta_2}{\beta_1-\beta_2}}\\
                & \, \Bigg(\frac{\heat(r_1,p_1,s+2\sigma-s_1,d)^{\beta_1}\heat(q_1,p_2,s+2\sigma-s_2,d)^{1-\beta_1}}{\beta_1-\theta}\Bigg)^{\frac{\beta_1-\theta}{\beta_1-\beta_2}} \, .
            \end{split}
        \end{align}
        
        \begin{theorem}[Gagliardo-Nirenberg-Sobolev]\label{thm-GNS}
            Let $f\in C^\infty_c(\R^d)$. Assume that $p,p_1,p_2\in[1,\infty]$, $s,s_1,s_2\in\R$ satisfy \eqref{ass-para} and recall the definition of $\theta\in(0,1)$ from \eqref{def-theta}. Then we have that
            \begin{align}
                \begin{split}
                \MoveEqLeft\||\nabla|^s f\|_p \, \leq \, \GNS(p,p_1,p_2,s,s_1,s_2,d)\| |\nabla|^{s_1} f\|_{p_1}^{\theta}\| |\nabla|^{s_2}  f\|_{p_2}^{1-\theta} \, .
                \end{split}
            \end{align}
        \end{theorem}

        \begin{remark}
            The constants we obtain here are not sharp. Instead, Theorem \ref{thm-GNS} provides orders of magnitude for the GNS constants, which can be useful, e.g., in a priori estimates. Our estimates do not allow for determining the endpoint estimates $\theta=0,1$, which would correspond to Sobolev inequalities, see Remark \ref{rem-sob}. As explained above, the endpoint cases exhibit peculiar behavior, requiring a deeper analysis with better estimates than those presented here.
        \end{remark}

    \section{Ideas of the proof}

    \subsection{Interpolation via heat map} 
    Let us assume the validity of Theorem \ref{thm-heat-ST}. Our goal is to prove Theorem \ref{thm-GNS}. 
    \par Let $s,s_1,s_2\in \R$, and $p,p_1,p_2 \in [1,\infty]$ satisfy \eqref{ass-para}, $\theta\in(0,1)$ be given by \eqref{def-theta}, and assume $f\in C^\infty_c(\R^d)$. The starting point for our approach is the representation of the inverse Laplacian via the heat-map. Let $\sigma>0$. Then we have 
        \begin{align}\label{id-inv-lapl}
            |\nabla|^{-2\sigma} f(x) \, = \, \frac1{\Gamma(\sigma)}\int_0^\infty dt \, t^{\sigma-1} e^{t\Delta} f (x) \, ,
        \end{align}
    see, \cite{balakrishnan60,kato-frac-power-61,komatsu66,stinga-phd,stinga-handbook-2019,stein-sing-diff-70,Yosida-FuncAna}, and also \cite[Corollary 15.20]{kunstmannweis-maximal}. Fix some $t_0>0$ to be determined below. Minkowski's inequality then yields
        \begin{align}    \begin{split}
            \MoveEqLeft\||\nabla|^s f\|_p \, = \, \||\nabla|^{-2\sigma}|\nabla|^{s+2\sigma} f\|_p\\
            \leq & \, \frac1{\Gamma(\sigma)}\Big(\int_0^{t_0} dt \, t^{\sigma-1} \||\nabla|^{s+2\sigma}e^{t\Delta} f \|_p \,+ \, \int_{t_0}^\infty dt \, t^{\sigma-1} \||\nabla|^{s+2\sigma}e^{t\Delta} f\|_p\Big) \, .
        \end{split}
    \end{align} 
    Next, we apply H\"older to obtain
    \begin{align} 
        \begin{split}
            \MoveEqLeft\||\nabla|^s f\|_p \\
            \leq & \, \frac1{\Gamma(\sigma)}\Big(\int_0^{t_0} dt \, t^{\sigma-1} \||\nabla|^{s+2\sigma}e^{t\Delta} f \|_{q_2}^{\beta_2}\||\nabla|^{s+2\sigma}e^{t\Delta} f \|_{r_2}^{1-\beta_2} \\
            & + \, \int_{t_0}^\infty dt \, t^{\sigma-1} \||\nabla|^{s+2\sigma}e^{t\Delta} f \|_{r_1}^{\beta_1}\||\nabla|^{s+2\sigma}e^{t\Delta} f \|_{q_1}^{1-\beta_1}\Big)
        \end{split}
    \end{align}
    for appropriate choices of $\beta_j$, $r_j$, $q_j$ satisfying
    \begin{equation}
        \frac1p \, = \, \frac{\beta_1}{r_1} \, + \, \frac{1-\beta_1}{q_1} \, = \, \frac{1-\beta_2}{r_2} \, + \, \frac{\beta_2}{q_2} \, .
    \end{equation}
    At this stage, we need to restrict to $(\beta_1,\beta_2,r_1,r_2,q_1,q_2,\sigma)\in\Sigma$ in order to employ Theorem \ref{thm-heat-ST}. For $\Sigma$ to be non-empty, we recover \eqref{ass-para}. Then we to obtain 
    \begin{align} 
        \begin{split}
            \MoveEqLeft\||\nabla|^s f\|_p \\
            \leq & \, C_1\||\nabla|^{s_1}f\|_{p_1}^{\beta_2}\||\nabla|^{s_2}f\|_{p_2}^{1-\beta_2} \, \int_0^{t_0} dt \, t^{-1+(\theta-\beta_2)\frac{d}2(\frac1{p_1}-\frac1{p_2}-\frac{s_1-s_2}d)} \\
            & \, + \, C_2\||\nabla|^{s_1}f\|_{p_1}^{\beta_1}\||\nabla|^{s_2}f\|_{p_2}^{1-\beta_1} \int_{t_0}^\infty dt \, t^{-1-(\beta_1-\theta)\frac{d}2(\frac1{p_1}-\frac1{p_2}-\frac{s_1-s_2}d)} \, .
        \end{split}
    \end{align}
    These integrals are finite due to \eqref{ass-para}. Then we find that
    \begin{align} 
        \begin{split}
            \||\nabla|^s f\|_p 
            \leq & \, C_1\||\nabla|^{s_1}f\|_{p_1}^{\beta_2}\||\nabla|^{s_2}f\|_{p_2}^{1-\beta_2} \, t_0^{(\theta-\beta_2)\frac{d}2(\frac1{p_1}-\frac1{p_2}-\frac{s_1-s_2}d)} \\
            & \, + \, C_2\||\nabla|^{s_1}f\|_{p_1}^{\beta_1}\||\nabla|^{s_2}f\|_{p_2}^{1-\beta_1} t_0^{-(\beta_1-\theta)\frac{d}2(\frac1{p_1}-\frac1{p_2}-\frac{s_1-s_2}d)} \, .
        \end{split}
    \end{align}
    Then Theorem \ref{thm-GNS} follows from first equating the coefficients by choosing $t_0$ appropriately, and then minimizing w.r.t. to the free parameters $\beta_j,r_j,\sigma$.
    
    \subsection{Quantitative parabolic estimates}

    We now turn to establishing an upper bound on the constants $\heat(p,r,s,d)$ in Theorem \ref{thm-heat-ST}. The ideas presented here are standard, and are only repeated to obtain explicit upper bounds. 
    \par Recall that the heat kernel is given by
    \begin{equation}\label{def-heat-kernel}
        e^{t\Delta}(x,y) \, = \, G_t(x-y) \, := \, (4\pi t)^{-\frac{d}2} e^{-\frac{|x-y|^2}{4t}} \, .
    \end{equation}
    Our strategy is to establish the following inequalities
    \begin{align}
        \|(-\Delta)^n e^{t\Delta}f\|_p \, &\leq\, A_1(p,r,n,d)t^{-n-\frac{d}2(\frac1r-\frac1p)}\|f\|_r \, , \label{eq-dif-smoothing-form}\\
        \||\nabla|^{n+2\alpha} e^{t\Delta}f\|_p \, &\leq \, A_2(p,r,n,\alpha,d) t^{-\frac{n+2\alpha}2-\frac{d}2(\frac1r-\frac1p)}\|f\|_r  \, , \label{eq-frac-smoothing-form}\\
        \||\nabla|^{-s} e^{t\Delta}f\|_p \, &\leq \, A_3(p,r,s,d)t^{\frac{s}2-\frac{d}2(\frac1r-\frac1p)} \|f\|_r \, , \label{eq-inv-smoothing-form}
    \end{align}
    where $1\leq r\leq p \leq \infty$, $\alpha\in(0,1)$, $n\in\N_0$, $s>0$, $t>0$, and, in \eqref{eq-inv-smoothing-form}, 
    \begin{equation}
        \frac1r \, > \, \frac1p+\frac{s}d \, .
    \end{equation} 
    Theorem \ref{thm-heat-ST} then follows from \eqref{eq-dif-smoothing-form}--\eqref{eq-inv-smoothing-form}.

    \section{Proof of Theorem \ref{thm-heat-ST}}
    
    \begin{lemma}[Proof of \eqref{eq-dif-smoothing-form}]\label{lem-int-smoothing}
        Let $f\in C^\infty_c(\R^d)$, $1\leq r \leq p \leq \infty$, $t>0$. Let $q$ be given by
        \begin{equation}
            \frac1q+\frac1r=1+\frac1p \, .
        \end{equation}
        Then we have that
        \begin{equation}
            \|e^{t\Delta}f\|_p \, \leq\, \frac{\young(p,q,r,d)}{(4\pi)^{\frac{d}2(\frac1r-\frac1p)}q^{\frac{d}{2q}}}t^{-\frac{d}2(\frac1r-\frac1p)}\|f\|_r \, .
        \end{equation}
        In addition, we have that
        \begin{equation}
            \|G_t\|_q \, = \, \frac1{(4\pi t)^{\frac{d}{2q'}}q^{\frac{d}{2q}}} \, .
        \end{equation}
    \end{lemma}
    \begin{proof}
        Using Young's inequality \eqref{eq-young}, we have that
        \begin{equation} \label{eq-heat-young-0}
            \|e^{t\Delta}f\|_p \, \leq \, \young(p,q,r,d) \|G_t\|_q\|f\|_r \, ,
        \end{equation}
        where $\frac1q+\frac1r=1+\frac1p$, and
        Using \eqref{def-heat-kernel}, we have that
        \begin{align} \label{eq-heat-Lq-norm}
            \begin{split}
            \|G_t\|_q \, &= \, \frac1{(4\pi t)^{d/2}}\Big(\int dx \, e^{-\frac{q|x|^2}{4t}}\Big)^{\frac1q}\\
            &= \, \frac1{(4\pi t)^{\frac{d}{2q'}}q^{\frac{d}{2q}}} \, , 
            \end{split}
        \end{align}
        whenever $q\in[1,\infty)$. In addition, we have the uniform estimate
        \begin{equation}\label{eq-heat-Linfty-norm}
            \|G_t\|_\infty \, = \, \frac1{(4\pi t)^{\frac{d}2}} \, ,
        \end{equation}
        such that, by recalling that $\infty^{1/\infty}=1$ by convention, we obtain 
        \begin{equation} \label{eq-heat-all-Lq-norm}
            \|G_t\|_q \, = \, \frac1{(4\pi t)^{\frac{d}{2q'}}q^{\frac{d}{2q}}} 
        \end{equation}
        for all $q\in[1,\infty]$. Combining \eqref{eq-heat-all-Lq-norm} with \eqref{eq-heat-young-0} and \eqref{def-young-const}, we finish the proof.
    \end{proof}
    
    \begin{lemma}\label{lem-dif-smoothing-form}
        Let $f\in C^\infty_c(\R^d)$, $1\leq p \leq \infty$, $n\in\N$, and $t>0$. Let $\Gamma$ denote the Gamma function. Then we have that
        \begin{equation}\label{eq-diff-smoothing-lem}
            \|(-\Delta)^n e^{t\Delta}f\|_p \, \leq\, \frac{\Gamma(\frac{d}2+n)2^n}{\Gamma(\frac{d}2)} t^{-n}\|f\|_p \, .
        \end{equation}
        Using Lemma \ref{lem-int-smoothing}, we can extend \ref{eq-diff-smoothing-lem} to include $n=0$.
    \end{lemma}
    \begin{proof}
        We proceed as in the proof of Lemma \ref{lem-int-smoothing} by employing Young's inequality to obtain
        \begin{equation}
            \|(-\Delta)^n e^{t\Delta}f\|_p \, \leq \, \|(-\Delta)^n G_t\|_1 \|f\|_p \, .
        \end{equation}
        We are thus tasked with bounding $\|(-\Delta)^n G_t\|_1$. 
        \par We start with the observation that
        \begin{equation} \label{eq-heat-space-time-equiv}
            (-\Delta)^ne^{t\Delta} f \, = \, \big((-\partial_t)^n G_t\big)*f \, ,
        \end{equation}
        and that
        \begin{align} \label{eq-partialt-heat-1}
            \begin{split}
                \Big((-\partial_t)^nG_t(x)\Big) \, =& \, (-\partial_t)^n\Big((4\pi t)^{-\frac{d}2} e^{-\frac{|x|^2}{4t}}\Big) \\
                =& \, \frac{(-1)^n}{(4\pi)^{\frac{d}2}}\sum_{k=0}^n\binom{n}{k} \frac{\Gamma(\frac{d}2+k)}{\Gamma(\frac{d}2)}(-1)^kt^{-\frac{d}2-k}\partial_t^{n-k}\big(e^{-\frac{|x|^2}{4t}}\big) \, ,
            \end{split}
        \end{align}
        where we used the Leibniz formula in the second step. Let 
        \begin{equation} \label{def-R-fdb}
            R(\ell) \, := \, \Big\{\vec{r}_\ell\in \N_0^\ell \mid \sum_{j=1}^\ell jr_j \, = \, \ell \Big\}
        \end{equation}
        Using the Fa\`a di Bruno formula, we obtain that
        \begin{align} \label{eq-heat-fdb}
            \begin{split}
            \MoveEqLeft\partial_t^{n-k}\big(e^{-\frac{|x|^2}{4t}}\big) \\
            =& \, (n-k)!e^{-\frac{|x|^2}{4t}}\sum_{\substack{\vec{r}_{n-k}\in\\ R(n-k)}}\prod_{j=1}^{n-k}\frac1{r_j!}\Big[\frac1{j!}\partial_t^j\Big(-\frac{|x|^2}{4t}\Big)\Big]^{r_j}\\
            =& \, (n-k)!e^{-\frac{|x|^2}{4t}}\sum_{\substack{\vec{r}_{n-k}\in\\ R(n-k)}}\Big(-\frac{|x|^2}4\Big)^{|\vec{r}_{n-k}|_1}\\
            & \, \prod_{j=1}^{n-k}\frac1{r_j!}\Big(\frac{(-1)^j}{t^{j+1}}\Big)^{r_j} \\
            =& \, \Big(\frac{-1}{t}\Big)^{n-k}(n-k)!e^{-\frac{|x|^2}{4t}}\sum_{\substack{\vec{r}_{n-k}\in\\ R(n-k)}}\Big(-\frac{|x|^2}{4t}\Big)^{|\vec{r}_{n-k}|_1}\prod_{j=1}^{n-k}\frac1{r_j!} \, ,
            \end{split}
        \end{align}
        where we employed the summation condition in \eqref{def-R-fdb} and where 
        \begin{equation}
            |\vec{r}_{n-k}|_1:=\sum_{j=1}^{n-k}r_j
        \end{equation} 
        denotes the $\ell^1$-norm. Recalling \eqref{eq-heat-space-time-equiv}, and plugging 
        \eqref{eq-heat-fdb} into \eqref{eq-partialt-heat-1} yields
        \begin{align}
            \begin{split}
                \MoveEqLeft\Big\|(-\Delta)^nG_t\Big\|_1 \\
                \leq & \, \frac{n!}{(4\pi)^{\frac{d}2}t^{\frac{d}2+n}}\sum_{k=0}^n\frac{\Gamma(\frac{d}2+k)}{k!\Gamma(\frac{d}2)}\sum_{\substack{\vec{r}_{n-k}\in\\ R(n-k)}}\int dx \, \Big(\frac{|x|^2}{4t}\Big)^{|\vec{r}_{n-k}|_1} e^{-\frac{|x|^2}{4t}}\prod_{j=1}^{n-k}\frac1{r_j!} \, .
            \end{split}
        \end{align}
        Substituting $y:=x/(2\sqrt{t})$, we can simplify the expression as
        \begin{align} \label{eq-partialt-heat-L1-bd-1}
            \begin{split}
                \MoveEqLeft\Big\|(-\Delta)^nG_t\Big\|_1 \\
                \leq & \, \frac{n!}{(\pi)^{\frac{d}2}t^n}\sum_{k=0}^n\frac{\Gamma(\frac{d}2+k)}{k!\Gamma(\frac{d}2)}\sum_{\substack{\vec{r}_{n-k}\in\\ R(n-k)}}\int dy \, |y|^{2|\vec{r}_{n-k}|_1} e^{-|y|^2}\prod_{j=1}^{n-k}\frac1{r_j!} \, .
            \end{split}
        \end{align}
        Recall that the Gaussian satisfies
        \begin{align}
        \begin{split}
            \int dy \, |y|^{2\ell} e^{-|y|^2} \, &= \, (-\partial_\lambda)^\ell\Big|_{\lambda=1}\int dy \, e^{-\lambda|y|^2}\\
            &= \, (-\partial_\lambda)^\ell\Big|_{\lambda=1}\Big(\frac{\pi}{\lambda}\Big)^{\frac{d}2} \\
            &= \, \pi^{\frac{d}2}\frac{\Gamma(\frac{d}2+\ell)}{\Gamma(\frac{d}2)} \, .
        \end{split}
        \end{align}
        Abbreviating again with Gamma functions, \eqref{eq-partialt-heat-L1-bd-1} thus implies
        \begin{align} \label{eq-partialt-heat-L1-bd-2}
            \begin{split}
                \MoveEqLeft\Big\|(-\Delta)^nG_t\Big\|_1 \\
                \leq & \, \frac{n!}{t^n}\sum_{k=0}^n\frac{\Gamma(\frac{d}2+k)}{k!\Gamma(\frac{d}2)^2}\sum_{\substack{\vec{r}_{n-k}\in\\ R(n-k)}}\Gamma(\frac{d}2+|\vec{r}_{n-k}|_1)\prod_{j=1}^{n-k}\frac1{r_j!} \, .
            \end{split}
        \end{align}
        Recall the integral identity 
        \begin{equation}\label{eq-gamma-int-id}
            \Gamma(z) \, = \, \int_0^\infty d\sigma \, \sigma^{z-1}e^{-\sigma}
        \end{equation}
        for all $\Re(z)>0$. In order to evaluate \eqref{eq-partialt-heat-L1-bd-2}, we start by computing
        \begin{align}\label{eq-fdb-part-0}
            \begin{split}
                \sum_{\substack{\vec{r}_{n-k}\in\\ R(n-k)}}\Gamma\Big(\frac{d}2+|\vec{r}_{n-k}|_1\Big)\prod_{j=1}^{n-k}\frac1{r_j!}
                \, &=  \, \int_0^\infty d\sigma \, e^{-\sigma}\sum_{\substack{\vec{r}_{n-k}\in\\ R(n-k)}}\sigma^{d/2-1+\sum_{j=1}^{n-k}r_j} \prod_{j=1}^{n-k}\frac{1}{r_j!}\\
                &= \, \int_0^\infty d\sigma \, \sigma^{d/2-1} e^{-\sigma} \sum_{\substack{\vec{r}_{n-k}\in\\ R(n-k)}}\prod_{j=1}^{n-k}\frac{\sigma^{r_j}}{r_j!} \, .
            \end{split}
        \end{align}
        Given a real sequence $(x_j)_{j\in\N}\in\R^\N$, recall the definition of the (complete) exponential Bell polynomials
        \begin{equation}\label{def-Bell-polynomials}
            B_\ell(x_1,\ldots,x_\ell) \, := \, \ell! \sum_{\substack{\vec{r}_{\ell}\in\\ R(\ell)}}\prod_{j=1}^{\ell}\frac{1}{r_j!}\Big(\frac{x_j}{j!}\Big)^{r_j} \, .
        \end{equation}
        Define the generating function
        \begin{equation}\label{def-bell-gen-fcn}
            \Phi(\lambda,(x_j)_{j\in\N}) \, := \, \exp\Big(\sum_{j=1}^\infty\frac{x_j\lambda^j}{j!}\Big) \, . 
        \end{equation}
        Then the Fa\`a di Bruno formula yields that
        \begin{equation}\label{eq-bell-mgf-id}
            B_\ell(x_1,\ldots,x_\ell) \, = \, \partial_\lambda^\ell\Big|_{\lambda=0}\Phi(\lambda,(x_j)_{j\in\N}) \, .
        \end{equation}
        In order to calculate \eqref{eq-fdb-part-0}, we start by identifying 
        \begin{equation} \label{eq-bell-pol-id}
            \sum_{\substack{\vec{r}_{n-k}\in\\ R(n-k)}}\prod_{j=1}^{n-k}\frac{\sigma^{r_j}}{r_j!} \, = \, \frac1{(n-k)!}B_{n-k}(\sigma 1!,\ldots, \sigma n!) \, ,
        \end{equation}
        i.e., $x_j=\sigma j!$. Then the moment generating function is given by
        \begin{equation}\label{eq-mgf-1}
            \Phi(\lambda,(\sigma j!)_j) \, = \, \exp\Big(\sigma\sum_{j=1}^\infty \lambda^j\Big) \, = \, \exp\Big(\frac{\lambda \sigma}{1-\lambda}\Big) 
        \end{equation}
        for any $|\lambda|<1$. Collecting \eqref{eq-bell-pol-id}--\eqref{eq-mgf-1}, we hence obtain
        \begin{equation} \label{eq-fdb-part-2}
            \sum_{\substack{\vec{r}_{n-k}\in\\ R(n-k)}}\prod_{j=1}^{n-k}\frac{\sigma^{r_j}}{r_j!} \, = \, \frac1{(n-k)!}\partial_\lambda^{n-k}\Big|_{\lambda=0}\exp\Big(\frac{\lambda \sigma}{1-\lambda}\Big) \, .
        \end{equation}
        Plugging \eqref{eq-fdb-part-0} and \eqref{eq-fdb-part-2} into \eqref{eq-partialt-heat-L1-bd-2} implies
        \begin{align}\label{eq-partialt-heat-L1-bd-3}
            \begin{split}
                \MoveEqLeft\Big\|(-\Delta)^nG_t\Big\|_1 \\
                \leq & \, \frac{1}{t^n}\sum_{k=0}^n\binom{n}{k}\frac{\Gamma(\frac{d}2+k)}{\Gamma(\frac{d}2)^2}\int_0^\infty d\sigma \, \sigma^{d/2-1} e^{-\sigma}  \partial_\lambda^{n-k}\Big|_{\lambda=0}\exp\Big(\frac{\lambda \sigma}{1-\lambda}\Big) \, .
            \end{split}
        \end{align}
        Next, we employ the integral representation \eqref{eq-gamma-int-id} of $\Gamma$ again to obtain
        \begin{align}
            \begin{split}
                \MoveEqLeft\Big\|(-\Delta)^nG_t\Big\|_1 \\
                \leq & \, \frac{1}{t^n}\sum_{k=0}^n\binom{n}{k}\frac{1}{\Gamma(\frac{d}2)^2}\int_0^\infty d\sigma \, \sigma^{d/2-1} e^{-\sigma}  \partial_\lambda^{n-k}\Big|_{\lambda=0}\exp\Big(\frac{\lambda \sigma}{1-\lambda}\Big) \\
                & \, \int_0^\infty d\tau \, \tau^{\frac{d}2+k-1}e^{-\tau} \, .
            \end{split}
        \end{align}
        Using the Binomial Theorem, we can simplify this expression as
        \begin{align}\label{eq-partialt-heat-L1-bd-4}
            \begin{split}
                \MoveEqLeft\Big\|(-\Delta)^nG_t\Big\|_1 \\
                \leq & \, \frac{1}{t^n\Gamma(\frac{d}2)^2}\int_0^\infty d\sigma \,\sigma^{d/2-1} e^{-\sigma} \int_0^\infty d\tau \, \tau^{\frac{d}2-1}e^{-\tau}\\
                & \, (\tau+\partial_\lambda)^n\Big|_{\lambda=0}\exp\Big(\frac{\lambda \sigma}{1-\lambda}\Big)  \, .
            \end{split}
        \end{align}
        Using the fact that
        \begin{equation}
            (\tau+\partial_\lambda)^n\Big|_{\lambda=0}\exp\Big(\frac{\lambda \sigma}{1-\lambda}\Big) \, = \, \partial_\lambda^n\Big|_{\lambda=0}\exp\Big(\frac{\lambda \sigma}{1-\lambda}+\lambda \tau \Big) \, ,
        \end{equation}
        pulling $\partial_\lambda^n\Big|_{\lambda=0}$ out of the integrals, \eqref{eq-partialt-heat-L1-bd-4} implies 
        \begin{align}
            \begin{split}
                \MoveEqLeft\Big\|(-\Delta)^nG_t\Big\|_1 \\
                \leq & \, \frac{1}{t^n\Gamma(\frac{d}2)^2}\partial_\lambda^n\Big|_{\lambda=0}\Big[\int_0^\infty d\sigma \,\sigma^{d/2-1} \exp\Big(-\frac{1-2\lambda}{1-\lambda}\sigma\Big) \\
                & \, \int_0^\infty d\tau \, \tau^{\frac{d}2-1}e^{-(1-\lambda)\tau}\Big] \, .
            \end{split}
        \end{align}
        Substituting $\sigma\to (1-\lambda)/(1-2\lambda)\sigma$ and $\tau\to 1/(1-\lambda) \tau$, and recalling \eqref{eq-gamma-int-id}, we obtain
        \begin{align}
            \begin{split}
                \MoveEqLeft\Big\|(-\Delta)^nG_t\Big\|_1 \\
                \leq & \, \frac{1}{t^n}\partial_\lambda^n\Big|_{\lambda=0}\Big[\Big(\frac{1-\lambda}{1-2\lambda}\Big)^{\frac{d}2}\frac1{(1-\lambda)^{\frac{d}2}}\Big] \\
                =& \, \frac{\Gamma(\frac{d}2+n)2^n}{t^n\Gamma(\frac{d}2)} \, .
            \end{split}
        \end{align}
        This concludes the proof.
    \end{proof}

    \begin{proposition}[Proof of \eqref{eq-dif-smoothing-form}]\label{prop-par-est-int}
        Assume $f\in C^\infty_c(\R^d)$, $s\in 2\N_0$, and $1\leq r\leq p\leq \infty$. Let $q\in[1,\infty]$ be given by
        \begin{equation}    \frac1q+\frac1r \, = \, 1 +\frac1p \, .
        \end{equation}
        If $r<p$, we have
        \begin{align} \label{eq-par-est-int}
            \begin{split}
                \MoveEqLeft \||\nabla|^s e^{t\Delta}f\|_p \\
                \leq & \, \frac{\young(p,q,r,d)2^s}{(4\pi)^{\frac{d}2(\frac1r-\frac1p)}q^{\frac{d}{2q}}(q')^{\frac{d}{2q'}}}\frac{\Gamma(\frac{d+s}2)}{\Gamma(\frac{d}2)s^{\frac{s}2}} \Big[\frac{s}2+\frac{d}2\Big(\frac1r-\frac1p\Big)\Big]^{\frac{s}2+\frac{d}2(\frac1r-\frac1p)}\\
                & \, t^{-\frac{s}2-\frac{d}2(\frac1r-\frac1p)}\|f\|_r \, ,
            \end{split}
        \end{align}
        and if $r=p$, we have
        \begin{align}
            \begin{split}
                \MoveEqLeft \|(-\Delta)^{\frac{s}2} e^{t\Delta}f\|_p \\
                \leq & \, \young(p,1,p,d)\frac{\Gamma(\frac{d+s}2)2^{\frac{s}2}}{\Gamma(\frac{d}2)}  t^{-\frac{s}2}\|f\|_p \, .
            \end{split}
        \end{align}
        $r=p$ coincides with the the limit $\lim_{r\uparrow p}$ in \eqref{eq-par-est-int}.
    \end{proposition}
    \begin{proof}
        Fix $\lambda\in[0,1]$, which we will determine below. Applying Lemma \ref{lem-dif-smoothing-form} yields
    \begin{align}
        \begin{split}
            \|(-\Delta)^{\frac{s}2} e^{t\Delta}f\|_p \, \leq\, \frac{\Gamma(\frac{d+s}2)2^{\frac{s}2}}{\Gamma(\frac{d}2)} [(1-\lambda)t]^{-\frac{s}2}\|e^{\lambda t\Delta}f\|_p \, .
        \end{split}
    \end{align}
    Next, we employ Lemma \ref{lem-int-smoothing} to obtain
    \begin{align} \label{eq-heat-ST-s=0-proof}
        \begin{split}
            \MoveEqLeft \|(-\Delta)^{\frac{s}2} e^{t\Delta}f\|_p \\
            \leq & \, \frac{\young(p,q,r,d)}{(4\pi)^{\frac{d}2(\frac1r-\frac1p)}q^{\frac{d}{2q}}}\frac{\Gamma(\frac{d+s}2)2^{\frac{s}2}}{\Gamma(\frac{d}2)} (1-\lambda)^{-\frac{s}2}\lambda^{-\frac{d}2(\frac1r-\frac1p)}\\
            & \, t^{-\frac{s}2-\frac{d}2(\frac1r-\frac1p)}\|f\|_r \, ,
        \end{split}
    \end{align}
    where $q\in[1,\infty]$ is given by
    \begin{equation}\label{eq-q-cond-proof}
        \frac1q+\frac1r \, = \, 1 +\frac1p \, .
    \end{equation}
    In case $r<p$, minimizing w.r.t. $\lambda\in[0,1]$ using Lemma \ref{lem-inv-pow-min} then yields
    \begin{align}\label{eq-heat-ST-s=0-proof-1}
        \begin{split}
            \MoveEqLeft \|(-\Delta)^{\frac{s}2} e^{t\Delta}f\|_p \\
            \leq & \, \frac{\young(p,q,r,d)}{(4\pi)^{\frac{d}2(\frac1r-\frac1p)}q^{\frac{d}{2q}}}\frac{\Gamma(\frac{d+s}2)2^{\frac{s}2}}{\Gamma(\frac{d}2)} \frac{\big(\frac{s}2+\frac{d}2(\frac1r-\frac1p)\big)^{\frac{s}2+\frac{d}2(\frac1r-\frac1p)}}{\big(\frac{s}2\big)^{\frac{s}2}\big(\frac{d}2(\frac1r-\frac1p)\big)^{\frac{d}2(\frac1r-\frac1p)}}\\
            & \, t^{-\frac{s}2-\frac{d}2(\frac1r-\frac1p)}\|f\|_r \, .
        \end{split}
    \end{align}
    Using \eqref{eq-q-cond-proof}, we write this in the slightly more compact form
    \begin{align}\label{eq-heat-ST-s=0-proof-2}
        \begin{split}
            \MoveEqLeft \|(-\Delta)^{\frac{s}2} e^{t\Delta}f\|_p \\
            \leq & \, \frac{\young(p,q,r,d)2^s}{(4\pi)^{\frac{d}2(\frac1r-\frac1p)}q^{\frac{d}{2q}}(q')^{\frac{d}{2q'}}}\frac{\Gamma(\frac{d+s}2)}{\Gamma(\frac{d}2)s^{\frac{s}2}} \Big[\frac{s}2+\frac{d}2\Big(\frac1r-\frac1p\Big)\Big]^{\frac{s}2+\frac{d}2(\frac1r-\frac1p)}\\
            & \, t^{-\frac{s}2-\frac{d}2(\frac1r-\frac1p)}\|f\|_r \, .
        \end{split}
    \end{align}
    If $r=p$, the minimum in \eqref{eq-heat-ST-s=0-proof} is attained at $\lambda=0$, yielding
    \begin{align}
        \begin{split}
            \MoveEqLeft \|(-\Delta)^{\frac{s}2} e^{t\Delta}f\|_p \\
            \leq & \, \young(p,1,p,d)\frac{\Gamma(\frac{d+s}2)2^{\frac{s}2}}{\Gamma(\frac{d}2)}  t^{-\frac{s}2}\|f\|_p \, .
        \end{split}
    \end{align}
    This finishes the proof.
    \end{proof}

    For the next statement, Let $q\in[1,\infty]$ be given by
    \begin{equation}    
        \frac1q+\frac1r \, = \, 1 +\frac1p \, .
    \end{equation}
    We abbreviate the constant obtained for $s\in 2\N_0$ in Proposition \ref{prop-par-est-int} according to \eqref{def-heat-int-r<p} and \eqref{def-heat-int-r=p}
    \begin{align} 
        \begin{split}
        \MoveEqLeft \heat(p,r,s,d) \\
        & = \, \frac{\young(p,q,r,d)2^s}{(4\pi)^{\frac{d}2(\frac1r-\frac1p)}q^{\frac{d}{2q}}(q')^{\frac{d}{2q'}}}\frac{\Gamma(\frac{d+s}2)}{\Gamma(\frac{d}2)s^{\frac{s}2}} \Big[\frac{s}2+\frac{d}2\Big(\frac1r-\frac1p\Big)\Big]^{\frac{s}2+\frac{d}2(\frac1r-\frac1p)} \, ,
        \end{split}
    \end{align}
    in case $r<p$, and, if $r=p$,
    \begin{equation}
        \heat(p,p,s,d) \, = \, \young(p,1,p,d)\frac{\Gamma(\frac{d+s}2)2^{\frac{s}2}}{\Gamma(\frac{d}2)} \, .
    \end{equation}
    
    \begin{proposition}[Proof of \eqref{eq-frac-smoothing-form}]\label{prop-par-est-frac}
        Assume $f\in C^\infty_c(\R^d)$, $s>0$ s.t. $\{s/2\}>0$, and $1\leq r\leq p\leq \infty$. 
        Then we have
        \begin{align} \label{eq-par-est-frac}
            \begin{split}
                \MoveEqLeft \|(-\Delta)^{\frac{s}2} e^{t\Delta}f\|_p \\
                \leq & \, \frac{\heat(p,r,2\floor{\frac{s}2}+2,d)\Gamma(\frac{s}2+\frac{d}2(\frac1r-\frac1p))}{\Gamma(\floor{\frac{s}2}+1+\frac{d}2(\frac1r-\frac1p))} t^{-\frac{s}2-\frac{d}2(\frac1r-\frac1p)}\|f\|_r \, .
            \end{split}
        \end{align}
    \end{proposition}
    \begin{proof}
    We start by recalling the definition of the fractional via the heat map, see e.g., \cite{balakrishnan60,kato-frac-power-61,komatsu66,stein-sing-diff-70,stinga-phd,stinga-handbook-2019,Yosida-FuncAna}. Let $f\in C^\infty_c(\R^d)$. For all $0<\alpha<1$, we have that 
    \begin{align}
        \begin{split}
            |\nabla|^{2\alpha}f & = \, (-\Delta)^\alpha f\\
            & = \, \frac1{\Gamma(-\alpha)}\int_0^\infty \frac{dt}{t^{1+\alpha}}\Big(e^{t\Delta}f-f\Big) \\
            &= \, \frac1{\Gamma(-\alpha)}\int_0^\infty \frac{dt}{t^{1+\alpha}}\int_0^t d\tau \, e^{\tau\Delta}\Delta f \, ,
        \end{split}
    \end{align}
    where $\Gamma$ denotes the Gamma function. Using Fubini, we thus obtain
    \begin{align} \label{id-frac-lap}
        \begin{split}
            |\nabla|^{2\alpha}f & = \, (-\Delta)^\alpha f\\
            &= \, \frac1{\Gamma(-\alpha)}\int_0^\infty dt \, e^{\tau\Delta}\Delta f\int_t^\infty \frac{d\tau}{\tau^{1+\alpha}} \\
            &= \, \frac1{\Gamma(1-\alpha)}\int_0^\infty \frac{dt}{t^\alpha} (-\Delta)e^{t\Delta} f \, ,
        \end{split}
    \end{align}
    where we employed the fact that $\Gamma(z+1)=z\Gamma(z)$ for all $\Re(z)>0$.
    \par Then \eqref{id-frac-lap} implies that
    \begin{align}
        \begin{split}
            \MoveEqLeft \|(-\Delta)^{\frac{s}2} e^{t\Delta}f\|_p \\
            &\leq \, \frac1{\Gamma(1-\{\frac{s}2\})}\int_0^\infty \frac{d\tau}{\tau^{\{\frac{s}2\}}} \|(-\Delta)^{\floor{\frac{s}2}+1}e^{(t+\tau)\Delta} f\|_p \, ,
        \end{split}
    \end{align}
    by the semi-group property of $e^{t\Delta}$. Proposition \ref{prop-par-est-int} then implies 
    \begin{align} \label{eq-heat-ST-proof>0-1}
        \begin{split}
            \MoveEqLeft \|(-\Delta)^{\frac{s}2} e^{t\Delta}f\|_p \\
            \leq & \, \frac{\heat(p,r,2\floor{\frac{s}2}+2,d)}{\Gamma(1-\{\frac{s}2\})}\|f\|_r \\
            & \, \int_0^\infty d\tau \, \tau^{-\{s/2\}} (t+\tau)^{-\floor{\frac{s}2}-1-\frac{d}2(\frac1r-\frac1p)} \, .
        \end{split}
    \end{align}
    In order to evaluate the integral, we substitute $\tau\to t\tau$ to obtain
    \begin{align}\label{eq-heat-ST-t-int-0}
        \begin{split}
        \MoveEqLeft\int_0^\infty d\tau \, \tau^{-\{s/2\}} (t+\tau)^{-\floor{\frac{s}2}-1-\frac{d}2(\frac1r-\frac1p)} \\
        & = \, t^{-\frac{s}2-\frac{d}2(\frac1r-\frac1p)}\int_0^\infty d\tau \, \tau^{-\{s/2\}} (1+\tau)^{-\floor{\frac{s}2}-1-\frac{d}2(\frac1r-\frac1p)} \, .
        \end{split}
    \end{align}
    By Lemma \ref{lem-beta}, we hence have that
    \begin{align}\label{eq-heat-ST-t-int-1}
        \begin{split}
        \MoveEqLeft\int_0^\infty d\tau \, \tau^{-\{s/2\}} (t+\tau)^{-\floor{\frac{s}2}-1-\frac{d}2(\frac1r-\frac1p)} \\
        & = \, \frac{\Gamma(\frac{s}2+\frac{d}2(\frac1r-\frac1p))\Gamma(1-\{\frac{s}2\})}{\Gamma(\floor{\frac{s}2}+1+\frac{d}2(\frac1r-\frac1p))}t^{-\frac{s}2-\frac{d}2(\frac1r-\frac1p)} \, .
        \end{split}
    \end{align}
    Plugging \eqref{eq-heat-ST-t-int-1} into \eqref{eq-heat-ST-proof>0-1} and rearranging factors yields
    \begin{align} 
        \begin{split}
            \MoveEqLeft \|(-\Delta)^{\frac{s}2} e^{t\Delta}f\|_p \\
            \leq & \, \frac{\heat(p,r,2\floor{\frac{s}2}+2,d)\Gamma(\frac{s}2+\frac{d}2(\frac1r-\frac1p))}{\Gamma(\floor{\frac{s}2}+1+\frac{d}2(\frac1r-\frac1p))} \, t^{-\frac{s}2-\frac{d}2(\frac1r-\frac1p)}\|f\|_r \, ,
        \end{split}
    \end{align}
    as we wanted to show.
    \end{proof}

    \begin{proposition}[Proof of \eqref{eq-inv-smoothing-form}]\label{prop-par-est-inv}
        Assume $f\in C^\infty_c(\R^d)$, $1\leq r< p\leq \infty$, and
        \begin{equation}\label{eq-s-inv-par-est-cond}
            0 \, < \, \sigma \, < \, \frac{d}2\Big(\frac1r-\frac1p\Big) \, .
        \end{equation}
        Let $q\in[1,\infty]$ be given by
        \begin{equation}
            1+\frac1p \, = \,\frac1q +\frac1r \, .
        \end{equation}
        Then we have
        \begin{align}
            \begin{split}
                \MoveEqLeft\||\nabla|^{-2\sigma}e^{t\Delta}f\|_p \, \leq \, \frac{\young(p,q,r,d)}{(4\pi)^{\frac{d}2(\frac1r-\frac1p)}q^{\frac{d}{2q}}}\frac{\Gamma(\frac{d}2(\frac1r-\frac1p)-\sigma)}{\Gamma(\frac{d}2(\frac1r-\frac1p))}t^{\sigma-\frac{d}2(\frac1r-\frac1p)}\|f\|_r \, .
            \end{split}
        \end{align}
    \end{proposition}
    \begin{proof}
        Using \eqref{id-inv-lapl}, we obtain
        \begin{equation}
            \||\nabla|^{-2\sigma}e^{t\Delta}f\|_p \, \leq\, \frac1{\Gamma(\sigma)}\int_0^\infty d\tau \, \tau^{\sigma-1} \|e^{(t+\tau)\Delta} f\|_p \, .
        \end{equation}
        Lemma \ref{lem-int-smoothing} then implies
        \begin{equation}
            \||\nabla|^{-2\sigma}e^{t\Delta}f\|_p \, \leq\, \frac{\young(p,q,r,d)}{(4\pi)^{\frac{d}2(\frac1r-\frac1p)}q^{\frac{d}{2q}}}\frac{\|f\|_r}{\Gamma(\sigma)}\int_0^\infty d\tau \, \tau^{\sigma-1}(t+\tau)^{-\frac{d}2(\frac1r-\frac1p)} \, .
        \end{equation}
        Substituting $\tau\to t\tau$ and applying Lemma \ref{lem-beta} using assumption \eqref{eq-s-inv-par-est-cond}, as in the proof of Proposition \ref{prop-par-est-frac}, yields
        \begin{equation}
            \||\nabla|^{-2\sigma}e^{t\Delta}f\|_p \, \leq\, \frac{\young(p,q,r,d)}{(4\pi)^{\frac{d}2(\frac1r-\frac1p)}q^{\frac{d}{2q}}}\frac{\Gamma(\frac{d}2(\frac1r-\frac1p)-\sigma)}{\Gamma(\frac{d}2(\frac1r-\frac1p))}t^{\sigma-\frac{d}2(\frac1r-\frac1p)}\|f\|_r \, ,
        \end{equation}
        which is what we intended to prove.
    \end{proof}

    \section{Proof of Theorem \ref{thm-GNS}}
    
        Let $s,s_1,s_2\in \R$, and $p,p_1,p_2 \in [1,\infty]$ satisfy \eqref{ass-para}, $\theta\in(0,1)$ be given by \eqref{def-theta}, and assume $f\in C^\infty_c(\R^d)$. Let $\sigma>0$ satisfy
        \begin{equation}\label{ass-sigma}
            \sigma \, > \, \frac{s_2-s}2 \, - \, \frac{d}{2p_2} \, .
        \end{equation}
        Let $t_0>0$ be arbitrary, to be determined below. Using Minkowski's inequality, \eqref{id-inv-lapl} implies
        \begin{align} \label{eq-GNS-decomp}
            \begin{split}
            \MoveEqLeft\||\nabla|^s f\|_p \, = \, \||\nabla|^{-2\sigma}|\nabla|^{s+2\sigma} f\|_p\\
            \leq & \, \frac1{\Gamma(\sigma)}\Big(\int_0^{t_0} dt \, t^{\sigma-1} \||\nabla|^{s+2\sigma}e^{t\Delta} f \|_p \,+ \, \int_{t_0}^\infty dt \, t^{\sigma-1} \||\nabla|^{s+2\sigma}e^{t\Delta} f\|_p\Big) \, .
        \end{split}
    \end{align}
    Now fix some $\beta_1\in(\theta,1]$ and $\beta_2\in[0,\theta)$. As a consequence of \eqref{ass-para} and \eqref{ass-sigma}, there exist $r_1,r_2\in[1,\infty]$ such that
    \begin{align}
        \frac1p \, - \, (1-\beta_1)\Big(\frac1{p_{2}}-\frac{s_{2}-s-2\sigma}d\Big) \, < \, \frac{\beta_1}{r_1} \, < \, \beta_1\Big(\frac1{p_1}-\frac{s_1-s-2\sigma}d\Big) \, \label{ass-r1}\\
        \frac1p \, - \, \beta_2\Big(\frac1{p_1}-\frac{s_1-s-2\sigma}d\Big) \, < \, \frac{1-\beta_2}{r_2} \, < \, (1-\beta_2)\Big(\frac1{p_2}-\frac{s_2-s-2\sigma}d\Big) \, , \label{ass-r2}
    \end{align}
    unless $\beta_2=1-\beta_1=0$. If $\beta_2=1-\beta_1=0$, we choose $r_j=p$, and have by assumptions \eqref{ass-para} and \eqref{ass-sigma} that 
    \begin{align} \label{eq-beta2=1-beta1=0}
        \frac1{r_j}\, = \, \frac1p \, < \, \frac1{p_j} - \frac{s_j-s-2\sigma}d \, .
    \end{align}
    \par In fact, by assumptions \eqref{ass-para} and \eqref{ass-sigma}, we have that
    \begin{align}
            \frac1{p_j}-\frac{s_j-s-2\sigma}d \, &= \, \Big(\frac1{p_j}-\frac1p-\frac{s_j-s}d\Big) \, + \, \frac1p \, + \,\frac{2\sigma}d \, > \, 0 \, .
    \end{align}
    In addition, employing \eqref{def-theta}, we have that
    \begin{align}
        \begin{split}
        \MoveEqLeft \frac1p \, - \, (1-\beta_1)\Big(\frac1{p_{2}}-\frac{s_{2}-s-2\sigma}d\Big) \, < \, \beta_1\Big(\frac1{p_1}-\frac{s_1-s-2\sigma}d\Big) \\
        \Leftrightarrow & \, (\beta_1-\theta)\Big(\frac1{p_1}-\frac1{p_2}-\frac{s_1-s_2}d\Big) \, > \, 0 \, ,
        \end{split}
    \end{align}
    which is satisfied by assumptions \eqref{ass-para} and $\beta_1>\theta$, and analogously
    \begin{align}
        \begin{split}
        \MoveEqLeft \frac1p \, - \, \beta_2\Big(\frac1{p_1}-\frac{s_1-s-2\sigma}d\Big) \, < \, (1-\beta_2)\Big(\frac1{p_2}-\frac{s_2-s-2\sigma}d\Big) \\
        \Leftrightarrow & \, (\theta-\beta_2)\Big(\frac1{p_1}-\frac1{p_2}-\frac{s_1-s_2}d\Big) \, > \, 0 \, ,
        \end{split}
    \end{align}
    due to \eqref{ass-para} and $\beta_2<\theta$.
    \par Let $q_1,q_2\in[1,\infty]$ be given by
    \begin{equation}\label{def-qj}
        \frac1p \, = \, \frac{\beta_1}{r_1} \, + \, \frac{1-\beta_1}{q_1} \, = \, \frac{1-\beta_2}{r_2} \, + \, \frac{\beta_2}{q_2} \, .
    \end{equation}
    In case that $\beta_2=1-\beta_1=0$, we have $r_j=p$ and choose $q_j\in[1,\infty]$ arbitrary.
    In particular, if $\beta_j\in(0,1)$, due to \eqref{ass-r1}, $q_1$ satisfies
    \begin{equation}\label{ass-q1}
        \frac1p \, - \, \beta_1\Big(\frac1{p_1}-\frac{s_1-s-2\sigma}d\Big) \, < \, \frac{1-\beta_1}{q_1} \, < \, (1-\beta_1)\Big(\frac1{p_2}-\frac{s_2-s-2\sigma}d\Big) \, ,
    \end{equation}
    and due to \eqref{ass-r2}, $q_2$ satisfies
    \begin{equation}\label{ass-q2}
        \frac1p \, - \, (1-\beta_2)\Big(\frac1{p_2}-\frac{s_2-s-2\sigma}d\Big) \, < \, \frac{\beta_2}{q_2} \, < \, \beta_2\Big(\frac1{p_1}-\frac{s_1-s-2\sigma}d\Big) \, .
    \end{equation}
    Using H\"older's inequality, \eqref{eq-GNS-decomp} then implies
    \begin{align} \label{eq-GNS-hoelder}
        \begin{split}
            \MoveEqLeft\||\nabla|^s f\|_p \\
            \leq & \, \frac1{\Gamma(\sigma)}\Big(\int_0^{t_0} dt \, t^{\sigma-1} \||\nabla|^{s+2\sigma}e^{t\Delta} f \|_{q_2}^{\beta_2}\||\nabla|^{s+2\sigma}e^{t\Delta} f \|_{r_2}^{1-\beta_2} \\
            & + \, \int_{t_0}^\infty dt \, t^{\sigma-1} \||\nabla|^{s+2\sigma}e^{t\Delta} f \|_{r_1}^{\beta_1}\||\nabla|^{s+2\sigma}e^{t\Delta} f \|_{q_1}^{1-\beta_1}\Big) \, .
        \end{split}
    \end{align}
    Using \eqref{ass-r1}, \eqref{ass-r2}, \eqref{ass-q1}, and \eqref{ass-q2}, and, in case $\beta_2=1-\beta_1=0$, \eqref{eq-beta2=1-beta1=0}, Theorem \ref{thm-heat-ST} then implies
    \begin{align}
        \begin{split}
            \MoveEqLeft\||\nabla|^s f\|_p \\
            \leq & \, \frac1{\Gamma(\sigma)}\Big(\heat(q_2,p_1,s+2\sigma-s_1,d)^{\beta_2}\heat(r_2,p_2,s+2\sigma-s_2,d)^{1-\beta_2}\\
            & \, \||\nabla|^{s_1}f\|_{p_1}^{\beta_2}\||\nabla|^{s_2}f\|_{p_2}^{1-\beta_2}\\
            & \, \int_0^{t_0} dt \, t^{\sigma-1-\beta_2[\frac{s+2\sigma-s_1}2+\frac{d}2(\frac1{p_1}-\frac1{q_2})]-(1-\beta_2)[\frac{s+2\sigma-s_2}2+\frac{d}2(\frac1{p_2}-\frac1{r_2})]} \\
            & + \, \heat(r_1,p_1,s+2\sigma-s_1,d)^{\beta_1}\heat(q_1,p_2,s+2\sigma-s_2,d)^{1-\beta_1}\\
            & \, \||\nabla|^{s_1}f\|_{p_1}^{\beta_1}\||\nabla|^{s_2}f\|_{p_2}^{1-\beta_1}\\
            & \int_{t_0}^\infty dt \, t^{\sigma-1-\beta_1[\frac{s+2\sigma-s_1}2+\frac{d}2(\frac1{p_1}-\frac1{r_1})]-(1-\beta_1)[\frac{s+2\sigma-s_2}2+\frac{d}2(\frac1{p_2}-\frac1{q_1})]}\Big) \, .
        \end{split}
    \end{align}
    Recalling the definition of $\theta$ \eqref{def-theta} and using \eqref{def-qj}, we can simplify this expression as
    \begin{align} 
        \begin{split}
            \MoveEqLeft\||\nabla|^s f\|_p \\
            \leq & \, \frac1{\Gamma(\sigma)}\Big(\heat(q_2,p_1,s+2\sigma-s_1,d)^{\beta_2}\heat(r_2,p_2,s+2\sigma-s_2,d)^{1-\beta_2}\\
            & \, \||\nabla|^{s_1}f\|_{p_1}^{\beta_2}\||\nabla|^{s_2}f\|_{p_2}^{1-\beta_2} \, \int_0^{t_0} dt \, t^{-1+(\theta-\beta_2)\frac{d}2(\frac1{p_1}-\frac1{p_2}-\frac{s_1-s_2}d)} \\
            & + \, \heat(r_1,p_1,s+2\sigma-s_1,d)^{\beta_1}\heat(q_1,p_2,s+2\sigma-s_2,d)^{1-\beta_1}\\
            & \, \||\nabla|^{s_1}f\|_{p_1}^{\beta_1}\||\nabla|^{s_2}f\|_{p_2}^{1-\beta_1} \int_{t_0}^\infty dt \, t^{-1-(\beta_1-\theta)\frac{d}2(\frac1{p_1}-\frac1{p_2}-\frac{s_1-s_2}d)}\Big) \, .
        \end{split}
    \end{align}
    The integrals are finite due to the assumptions \eqref{ass-para}, $\beta_1>\theta$, and $\beta_2<\theta$. In particular, we obtain
    \begin{align} \label{eq-GNS-par-est}
        \begin{split}
            \MoveEqLeft\||\nabla|^s f\|_p \\
            \leq & \, \frac1{\Gamma(\sigma)}\Big(\heat(q_2,p_1,s+2\sigma-s_1,d)^{\beta_2}\heat(r_2,p_2,s+2\sigma-s_2,d)^{1-\beta_2}\\
            & \, \||\nabla|^{s_1}f\|_{p_1}^{\beta_2}\||\nabla|^{s_2}f\|_{p_2}^{1-\beta_2} \,  \frac{t_0^{(\theta-\beta_2)\frac{d}2(\frac1{p_1}-\frac1{p_2}-\frac{s_1-s_2}d)}}{(\theta-\beta_2)\frac{d}2\big(\frac1{p_1}-\frac1{p_2}-\frac{s_1-s_2}d\big)} \\
            & + \, \heat(r_1,p_1,s+2\sigma-s_1,d)^{\beta_1}\heat(q_1,p_2,s+2\sigma-s_2,d)^{1-\beta_1}\\
            & \, \||\nabla|^{s_1}f\|_{p_1}^{\beta_1}\||\nabla|^{s_2}f\|_{p_2}^{1-\beta_1} \frac{t_0^{-(\beta_1-\theta)\frac{d}2(\frac1{p_1}-\frac1{p_2}-\frac{s_1-s_2}d)}}{(\beta_1-\theta)\frac{d}2\big(\frac1{p_1}-\frac1{p_2}-\frac{s_1-s_2}d\big)}\Big) \, .
        \end{split}
    \end{align}
    In particular, since we may choose $\beta_j\in(0,1)$, we may assume w.l.o.g. that $\||\nabla|^{s_2}f\|_{p_2}>0$, for otherwise both sides vanish. In order to equate the coefficients, we choose
    \begin{align} \label{eq-t0-choice}
        \begin{split}
        \MoveEqLeft t_0^{(\beta_1-\beta_2)\frac{d}2(\frac1{p_1}-\frac1{p_2}-\frac{s_1-s_2}d)} \\
        =& \, \Bigg(\frac{\||\nabla|^{s_1}f\|_{p_1}}{\||\nabla|^{s_2}f\|_{p_2}}\Bigg)^{\beta_1-\beta_2}\frac{\theta-\beta_2}{\beta_1-\theta}\\
        & \, \frac{\heat(r_1,p_1,s+2\sigma-s_1,d)^{\beta_1}\heat(q_1,p_2,s+2\sigma-s_2,d)^{1-\beta_1}}{\heat(q_2,p_1,s+2\sigma-s_1,d)^{\beta_2}\heat(r_2,p_2,s+2\sigma-s_2,d)^{1-\beta_2}} \, .
        \end{split}
    \end{align}
    Plugging this choice for $t_0$ into \eqref{eq-GNS-par-est} yields
    \begin{align} 
        \begin{split}
            \MoveEqLeft\||\nabla|^s f\|_p \\
            \leq & \, \frac4{d \, \Gamma(\sigma)\big(\frac1{p_1}-\frac1{p_2}-\frac{s_1-s_2}d\big)}\||\nabla|^{s_1}f\|_{p_1}^{\theta}\||\nabla|^{s_2}f\|_{p_2}^{1-\theta}\\
            & \, \Bigg(\frac{\heat(q_2,p_1,s+2\sigma-s_1,d)^{\beta_2}\heat(r_2,p_2,s+2\sigma-s_2,d)^{1-\beta_2}}{\theta-\beta_2}\Bigg)^{\frac{\theta-\beta_2}{\beta_1-\beta_2}}\\
            & \, \Bigg(\frac{\heat(r_1,p_1,s+2\sigma-s_1,d)^{\beta_1}\heat(q_1,p_2,s+2\sigma-s_2,d)^{1-\beta_1}}{\beta_1-\theta}\Bigg)^{\frac{\beta_1-\theta}{\beta_1-\beta_2}} \, .
        \end{split}
    \end{align}
    Finally, we minimize w.r.t. 
    \begin{equation}
        (\beta_1,\beta_2,r_1,r_2,\sigma)\in(\theta,1]\times[0,\theta)\times[1,\infty]^2\times(0,\infty)
    \end{equation}
    such that the constraints \eqref{ass-sigma}, \eqref{ass-r1}, \eqref{ass-r2} are satisfied, and where $q_j$ is given by \eqref{def-qj}, unless $\beta_2=1-\beta_1=0$. In the latter case, recall that we chose $q_j\in[1,\infty]$ arbitrary.
    
    \appendix

    \section{Helpful tools}

    \begin{lemma}\label{lem-inv-pow-min}
        Let $\alpha,\beta>0$. Then we have
        \begin{equation}
            \min_{\lambda\in[0,1]}\lambda^{-\alpha}(1-\lambda)^{-\beta} \, = \, \frac{(\alpha+\beta)^{\alpha+\beta}}{\alpha^\alpha\beta^\beta} \, .    
        \end{equation}
    \end{lemma}
    \begin{proof}
        We have that
        \begin{align}
            \frac{d}{d\lambda}\big(\lambda^{-\alpha}(1-\lambda)^{-\beta}\big) \, = \, \frac{-\alpha(1-\lambda)+\beta\lambda}{\lambda^{\alpha+1}(1-\lambda)^{\beta+1}} \, = \, \frac{(\alpha+\beta)\lambda -\alpha}{\lambda^{\alpha+1}(1-\lambda)^{\beta+1}} \, ,
        \end{align}
        which implies that the minimum is attained at $\lambda=\frac\alpha{\alpha+\beta}$. This concludes the proof.
    \end{proof}

    \begin{lemma}\label{lem-beta}
    Let $\alpha<1$ and $\beta>1-\alpha$. Then we have that
    \begin{equation}
        \int_0^\infty dx \, x^{-\alpha} (1+x)^{-\beta} \, = \, \frac{\Gamma(\alpha+\beta-1)\Gamma(1-\alpha)}{\Gamma(\beta)} \, .
    \end{equation}
    \end{lemma}
    \begin{proof}
        We start by recognizing that 
        \begin{equation}
            \int_0^\infty dx \, x^{-\alpha} (1+x)^{-\beta} \, = \, \int_0^\infty dx \, \Big(\frac{x}{1+x}\Big)^{-\alpha}\Big(\frac1{x+1}\Big)^{\alpha+\beta} \, .
        \end{equation}
        Substituting $\frac{x}{x+1}\to x$, we find that 
        \begin{equation}
            \int_0^\infty dx \, x^{-\alpha} (1+x)^{-\beta} \, = \, \int_0^1 dx \, x^{-\alpha}(1-x)^{\alpha+\beta-2} \, .
        \end{equation}
        Recognizing the Beta function, we thus have
        \begin{align}
            \begin{split}
                \int_0^\infty dx \, x^{-\alpha} (1+x)^{-\beta} \, =& \, B(1-\alpha,\alpha+\beta-1) \\
                =& \, \frac{\Gamma(\alpha+\beta-1)\Gamma(1-\alpha)}{\Gamma(\beta)} \, .
            \end{split}
        \end{align}
        With that, we finish the proof.
    \end{proof}

    \section{Proofs of error estimates\label{sec-error-est}}

    \begin{proof}[Proof of Proposition \ref{prop-riem-discr.}]
        We denote the Fourier transform
        \begin{equation}
            (\cF f)(p) \, := \,\hat{f}(p) \, := \, \int\dx{x} e^{-2\pi i p\cdot x}f(x) \, .
        \end{equation}
        Applying the Poisson summation formula, we then obtain that
        \begin{align}
            \label{eq-R-disc-PS}
            \begin{split}\vep^d \sum_{x\in\Z^d}f(\vep x)-\int_{\R^d}\dx{x}f(x) \, &= \, \sum_{p\in\Z^d}(\cF f(\vep \cdot))(p) - \int\dx{x}f(x)\\
            & = \, \sum_{p\in\Z^d\setminus\{0\}}\hat{f}\Big(\frac{p}\vep\Big) \, .
            \end{split}
        \end{align}
        Now observe that we have that
        \begin{equation}
            \Big| |p|^{d+\delta} \hat{f}\Big(\frac{p}\vep\Big)\Big| \, = \, \Big(\frac{\vep}{2\pi}\Big)^{d+\delta}\Big|\int\dx{x}e^{-2\pi ip\cdot x/\vep}|\nabla|^{d+\delta}f(x)\Big| \, \leq \, \Big(\frac{\vep}{2\pi}\Big)^{d+\delta}\|\nabla|^{d+\delta}f\|_1 \, .
        \end{equation}
        Employing this estimate and recalling \ref{def-lattice-const}, \eqref{eq-R-disc-PS} implies
        \begin{align}
            \Big|\vep^d \sum_{x\in\Z^d}f(\vep x)-\int_{\R^d}\dx{x}f(x)\Big| \, &= \, \Big|\sum_{p\in\Z^d\setminus\{0\}}\frac1{|p|^{d+\delta}} |p|^{d+\delta}\hat{f}\Big(\frac{p}\vep\Big)\Big|\\
            & \leq \, \Big(\frac{\vep}{2\pi}\Big)^{d+\delta}\Xi(d,\delta)\|\nabla|^{d+\delta}f\|_1 \, ,
        \end{align}
        concluding the proof.
    \end{proof}

    \begin{proof}[Proof of Lemma \ref{lem-lattice-const-bd}]
        We start by observing that we can write $\Z^d\setminus\{0\}$ as the disjoint union
        \begin{equation}
            \Z^d\setminus\{0\} \, = \, \bigcup_{k=1}^d\{\sigma(n,0) \mid n\in(\N\cup -\N)^k, \, \sigma\in \mathrm{Per}_d\} \, ,
        \end{equation}
        where $\mathrm{Per}_d$ is the set of all permutations acting on $d-$tupels $(n_1,\ldots,n_d)$. Consequently, we obtain that
        \begin{align}\label{eq-lattice-const-exp}
            \Xi(d,\delta) \, = \, \sum_{k=1}^d\binom{d}{k}2^k\sum_{n\in\N^k}\frac1{|n|^{d+\delta}} \, .
        \end{align}
        By the GM-QM inequality, we have that for all $n\in(\N\cup -\N)^k$
        \begin{equation}
            \sqrt{\frac{|n|^2}k} \, \geq \, \prod_{j=1}^k n_j^{\frac1k} \, .
        \end{equation}
        In particular, \eqref{eq-lattice-const-exp} implies that
        \begin{equation}
            \Xi(d,\delta) \, \leq \, \sum_{k=1}^d\binom{d}{k}2^k \frac1{k^{\frac{d+\delta}2}}\Big(\sum_{n\in\N}n^{-\frac{d+\delta}k}\Big)^k \, = \, \sum_{k=1}^d\binom{d}{k}2^k \frac1{k^{\frac{d+\delta}2}}\zeta\Big(\frac{d+\delta}k\Big)^k \, ,
        \end{equation}
        concluding the proof.
    \end{proof}

    \begin{proof}[Proof of Proposition \ref{prop-erg-thm}] 
        We adapt \cite[Proof of Prop. 3.5]{chenhott2023} to the present scenario. 
        \par We start by representing $f$ via the Fourier inversion theorem and obtain that
        \begin{align}
            \begin{split}\label{eq-erg-err-exp}
            \frac1n\sum_{k=1}^nf(x_0+k\alpha) - \int_0^1\dx{x}f(x) \, &= \, \sum_{p\in\Z}\hat{f}(p)\frac1n\sum_{k=1}^ne^{2\pi i p(x_0+k\alpha)}- \int_0^1\dx{x}f(x) \\
            & = \, \sum_{p\in\Z\setminus\{0\}}\hat{f}(p)\frac1n\sum_{k=1}^ne^{2\pi i p(x_0+k\alpha)} \, .
            \end{split}
        \end{align}
        We define the \emph{discrepancy function}
        \begin{equation}
            D_n(x) \, := \, \frac1n\sum_{k=1}^ne^{2\pi ikx} \, = \, e^{\pi i(n+1)x}\frac{\sin( n\pi x)}{n\sin(\pi x)} \, . 
        \end{equation}
        Now observe that we have that
        \begin{equation}
		      \frac{\dist(x,\Z)}{|\sin(\pi x)|} \, = \, \frac{\dist(x,\Z)}{|\sin\big(\pi \dist(x,\Z)\big)|} \, \leq\, \frac1{\pi}\sup_{0<y<\pi/2}\frac{y}{\sin(y)} \, \leq \, \frac12 \, .
	    \end{equation}
        Then \eqref{eq-erg-err-exp} implies that
        \begin{align}
            \Big|\frac1n\sum_{k=1}^nf(x_0+k\alpha) - \int_0^1\dx{x}f(x)\Big| \, &\leq \, \sum_{p\in\Z\setminus\{0\}}|\hat{f}(p)| |D_n(p\alpha)|\\
            & \leq \, \frac1{2n} \sum_{p\in\Z\setminus\{0\}}\frac{|\hat{f}(p)|}{\dist(p\alpha,\Z)}
        \end{align}
        Using now the Diophantine condition \eqref{def-diophantine}, we thus obtain
        \begin{align}
            \Big|\frac1n\sum_{k=1}^nf(x_0+k\alpha) - \int_0^1\dx{x}f(x)\Big| \, \leq \, \frac{\||\cdot|^{2+\sigma+\delta}\hat{f}\|_{\ell^1(\Z\setminus\{0\})}}{2Kn} \, ,
        \end{align}
        which finishes the proof.
    \end{proof}

    \bibliographystyle{plain}  % Here the bibliography 		     %
	\bibliography{references}  

\end{document}